\documentclass[manuscript, printscheme]{paper}

\usepackage[T2A]{fontenc}
\usepackage[tbtags]{amsmath}
\usepackage{amsfonts,amssymb}
\usepackage{amsthm}

\newtheorem{thm}{Theorem}[section]
\newtheorem{cor}[thm]{Corollary}
\newtheorem{lem}[thm]{Lemma}

 \newcommand{\RM}{\mathbb{R}}

 \newcommand{\ZM}{\mathbb{Z}}
 \newcommand{\CM}{\mathbb{C}}
  
 \newcommand{\dr}{\operatorname{dr}}
 
 \newcommand{\ds}{\operatorname{ds}}
 
 \newcommand{\dlambda}{\operatorname{d\lambda}}

\title{Fourier inversion theorems for integral transforms involving Bessel functions}

\author{Alexey Gorshkov}

\begin{document}

\maketitle

\begin{abstract} 
Usually such area of mathematics as differential equations acts as a consumer of results given by functional analysis. This article will give an example of the reverse interaction of these two fields of knowledge. Namely, the derivation and study of integral transforms will be carried out using partial differential equations. We will study generalised Weber-Orr transforms - its invertibility theorems in $f\in L_1\cap L_2$, spectral decomposition, Plancherel-Parseval identity. These transforms possess nontrivial kernel, so spectral decomposition must involve not only continuous spectrum, but also eigen functions which correspond to zero eigen value.  We give a new approach to the study of classical and generalized Weber-Orr transforms with a complete derivation of the inversion formulas.
\end{abstract}

{Primary 76D07; Secondary 33C10}

\tableofcontents

\section{Introduction}

Most of Fourier transforms are based on generalised eigen finctions associated with some differential operator $A$ and its inversion formulas represent the spectrum decomposition of $A$. If $A$ obtains complete system of generalised eigen functions and $\ker(A)=\{0\}$ then inversion property as well as Plancherel theorem become valid for corresponding Fourier transform.

But $A$ with nontrivial kernel breaks down all basic properties of the Fourier transform $F$. Thus transform $F$ also gets nontrivial kernel. In order to complete the system of generalised eigen functions one should add ordinary eigen functions $\{e_k\}$ associated with eigenvalue $\lambda=0$. And Plancherel equity doesn't hold until we complete it by the sum of squares of the Fourier coefficients gaining Plancherel-Parseval type identity like 
\begin{eqnarray}\label{intro:PP}
\|f\|^2 = \|F[f]\|^2+\sum_k (f,e_k)^2.
\end{eqnarray}

This article is devoted to transforms involving Bessel functions. Here we deduce invertibility theorems for Weber-Orr transforms even when some of them posess nontrivial kernel and prove Plancherel-Parseval identity above.  Applying presented methods one could obtain similar proofs for Hankel transform as well as for other ones involving Bessel functions.	

The inversion formulas for the Weber-Orr transforms, since the first publication of these transforms by Weber in 1873, have traditionally been given under the assumption of a limited variation for the original function $f$. In paper \cite{Tm} published in 1923 Titchmarsh proved inversion theorem under the assumption of bounded variation of $f$: 
\begin{equation}\label{intro:inv}
\frac{f(r-0)+f(r+0)}2 = W_{k,k}^{-1}\left [ W_{k,k}[f]\right ](r), r>r_0,
\end{equation}
where
\begin{align}\label{intro:weberorr}
W_{k,k}[f](\lambda) = \int_{r_0}^\infty \frac{J_{k}(\lambda s)Y_{k}(\lambda r_0) - Y_{k}(\lambda s)J_{k}(\lambda r_0)}{\sqrt{J_{k}^2(\lambda r_0) + Y_{k}^2(\lambda r_0)}} f(s) s \ds,k,k\in \RM,
\\ \label{int:weberorrinv}
W^{-1}_{k,k}[\hat f](r) = \int_{0}^\infty \frac{J_{k}(\lambda r)Y_{k}(\lambda r_0) - Y_{k}(\lambda r)J_{k}(\lambda r_0)}{\sqrt{J_{k}^2(\lambda r_0) + Y_{k}^2(\lambda r_0)}} \hat f (\lambda) \lambda \dlambda,
\end{align}
where $r_0>0$, $J_k(r)$, $Y_k(r)$ are Bessel functions of the first and second kind respectively.

Later, the definition of the Weber-Orr transforms was generalized to the case of various indices $k,l$:
\begin{align}\label{int:weberorr}
W_{k,l}[f](\lambda) = \int_{r_0}^\infty \frac{J_{k}(\lambda s)Y_{l}(\lambda r_0) - Y_{k}(\lambda s)J_{l}(\lambda r_0)}{\sqrt{J_{l}^2(\lambda r_0) + Y_{l}^2(\lambda r_0)}} f(s) s \ds,k,l\in \RM,
\\ \label{int:weberorrinv}
W^{-1}_{k,l}[\hat f](r) = \int_{0}^\infty \frac{J_{k}(\lambda r)Y_{l}(\lambda r_0) - Y_{k}(\lambda r)J_{l}(\lambda r_0)}{\sqrt{J_{l}^2(\lambda r_0) + Y_{l}^2(\lambda r_0)}} \hat f (\lambda) \lambda \dlambda.
\end{align}
	
Further studies of generalized Weber-Orr transforms were based on the works of Weber and Titchmarsh, and the restrictions on the  function $f$ as well as the equality in the form (\ref{intro:inv}) remained the same.

In this paper will be given new proofs of Weber-Orr transform including Plancherel type equities and the validity of the inversion formulas for $f\in L_1\cap L_2$ which is more typical for Fourier transforms. We give a new approach to the study of generalized and classical Weber-Orr transforms with a complete derivation of inversion formulas in the class of integrable functions where equality (\ref{intro:inv}) will be replaced by identity almost everywhere.

Weber-Orr transform is based on generalised eigen functions of operator 
$$
\Delta_k  = \frac 1r \frac {\partial}{\partial r}\left(r \frac {\partial}{\partial r}\right) - \frac{k^2}{r^2}.
$$
In unbounded domains with various boundary conditions it has continuous spectrum $\sigma(\Delta) = \RM_-$. Its resolvent $R(\tau) = (\Delta_k-\tau I)^{-1}$ is the branching function with jump discontinuity along spectrum at branch point $\tau=0$. In case of nontrivial kernel in addition to continuous part of spectrum the eigen value $\tau=0$ will generate term like $R_{-1}/\tau$ in expansion of $R(\tau)$, where $R_{-1}$ is the projector onto eigen subspace $\ker(\Delta_k)$ leading to equities like (\ref{intro:PP}). In this paper will be shown that namely $W_{k,k\pm 1}$ have a nontrivial kernel. 


The Weber-Orr transforms find an application in many fields of mathematics including analysis and mathematical physics.
With its help can be solved various boundary-value problems with Dirichlet condition. Of particular interest are the transforms $W_{k,k\pm 1}$. Using $W_{k,k-1}$ the author in \cite{G1} solved the non-stationary Stokes problem in exterior of the disc by explicit formula for vorticity in terms of Fourier coefficients $w_k(t,r)$:  
$$w_k(t,r) = W^{-1}_{|k|,|k|-1}\left [ e^{-\lambda^2 t} W_{|k|,|k|-1}[w(0,r)] \right ].$$
Then in paper \cite{G3} this integral transform was applied to 2D and 3D Navier-Stokes systems in external domains, so this type of transform can be considered as {\it hydrodynamic Fourier transform}. 

\section{Classical Weber-Orr Transform $W_{k,k}[\cdot]$}

In this section we prove inversion formula for $W_{k,k}[\cdot]$. The Fourier transform typically is based on generalized eigenfunctions of the Laplace operator $\Delta$ defined in an unbounded domain. $\Delta$ is a generator of a strongly continuous semigroup that solves a certain initial-boundary value problem for some partial differential equation. The transformation $W_{k,k}[\cdot]$ can be associated with the Dirichlet initial-boundary value problem for the heat equation in the exterior of a circle within radius $r_0>0$ with zero condition on its boundary. 

For functions defined on $E\subset \RM_+$ along with classic spaces $L_1(E)$, $L_2(E)$ we will use $L_2(E; r)$ of square-integrable functions with  infinitesimal element $r\dr$ supplied with the norm
$$
\|f\|^2_{L_2(E; r)}=\int\limits_E |f(r)|^2 r\dr.
$$

\begin{thm}\label{wkk:thm}Let $f(r) \sqrt r \in L_1(r_0,\infty)\cap L_2(r_0,\infty)$, $k \in \RM$, $r_0>0$. Then the  Weber-Orr transforms $W_{k,k}[\cdot]$ defined by (\ref{int:weberorr}), (\ref{int:weberorrinv}) satisfy  almost everywhere 
\begin{align}\label{wkk:inversion_formula}
f(r) = W^{-1}_{k,k}\left [W_{k,k} [f] \right ](r), 
\end{align}
and the following Plansherel equity holds:
$$
||f||_{L_2(r_0,\infty; r)} = \| W_{k,k}[f] \|_{L_2(0,\infty; r) } 
$$
\end{thm}

We derive the relation (\ref{wkk:inversion_formula}) from some initial boundary value problem for heat equation with the initial datum $f(r)$. 


Now we prove the following 
\begin{lem} \label{wkk:lemest} For $\lambda>1$ with some $C=C(k, r_0)$ holds
$$\left |W_{k,k} [f] \right (\lambda) | \leq \frac C {\sqrt \lambda} \|f \sqrt r \|_{L_1(0,\infty)}. $$ 
\end{lem}

\begin{proof}
Using 
\begin{align}\label{wkk:estJY}
 \frac{  \max \left(|J_{k}(\lambda r_0)|, |Y_{k}(\lambda r_0|)\right) }{\sqrt {{J_{k}(\lambda r_0)^2 + Y_{k}(\lambda r_0)^2}}}  \leq 1
\end{align}
and from the following asymptotic form of Bessel functions for large argument (see Bateman, Erdelyi\cite{BE})
\begin{align}\label{wkk:asymp1}
J_{k }(z)={\sqrt {\frac {2}{\pi z}}}\left(\cos \left(z-{\frac {k \pi }{2}}-{\frac {\pi }{4}}\right) - {(4k^2-1)(4k^2-9) \over 8z } \sin \left(z-{\frac {k \pi }{2}}-{\frac {\pi }{4}}\right) \right) 
   \\+\mathrm {O} \left({1 \over |z|^{2}} \right) \nonumber \\\label{wkk:asymp2}
   Y_{k }(z)={\sqrt {\frac {2}{\pi z}}}\left(\sin \left(z-{\frac {k \pi }{2}}-{\frac {\pi }{4}}\right) + {(4k^2-1)(4k^2-9) \over 8z } \cos \left(z-{\frac {k \pi }{2}}-{\frac {\pi }{4}}\right) \right) 
   \\+\mathrm {O} \left({1 \over |z|^{2}} \right) \nonumber 
\end{align}
we will have with some $C>0$
\begin{align*}
\left | \int_{r_0}^\infty
J_{k}(\lambda s) f(s) s\ds \right | =
\left |\frac 1{\sqrt {\lambda}} \int_{r_0}^\infty
J_{k}(\lambda s) \sqrt {\lambda  s}f(s)\sqrt s \ds \right | \leq \frac C {\sqrt \lambda} \|f \sqrt r \|_{L_1(0,\infty)},\\
\left | \int_{r_0}^\infty
Y_{k}(\lambda s) f(s) s\ds \right | =
\left |\frac 1{\sqrt {\lambda}} \int_{r_0}^\infty
Y_{k}(\lambda s) \sqrt {\lambda  s}f(s)\sqrt s \ds \right | \leq \frac C {\sqrt \lambda} \|f \sqrt r \|_{L_1(0,\infty)},
\end{align*}
and
$$
\left | \int_{r_0}^\infty \frac{J_{k}(\lambda s)Y_{k}(\lambda r_0) - Y_{k}(\lambda s)J_{k}(\lambda r_0)}{\sqrt{J_{k}^2(\lambda r_0) + Y_{k}^2(\lambda r_0)}} f(s) s \ds \right | \leq \frac {2C} {\sqrt \lambda} \|f \sqrt r \|_{L_1(0,\infty)}.
$$

The lemma is proved.

\end{proof}

\subsection{Initial boundary value problem}

Consider the equation
\begin{equation} \label{wkk:omegaeqpolar}
\frac{\partial w(t,r)}{\partial t} - \Delta_k w(t,r) = 0,
\end{equation}
where
$$
\Delta_k w(t,r) = \frac 1r \frac {\partial}{\partial r}\left(r \frac {\partial}{\partial r}w(t,r)\right) - \frac{k^2}{r^2} w(t,r).
$$

We supplement it with the Dirichlet boundary condition
\begin{equation} \label{wkk:omegaeqdirichlet}
w(t,r_0) = 0
\end{equation}
and the initial condition
\begin{equation} \label{wkk:omegaeqinit}
w(0,r) = f(r)
\end{equation}

For $k\in \ZM$ the problem (\ref{wkk:omegaeqpolar})-(\ref{wkk:omegaeqinit}) is the Dirichlet problem for the $k-$th Fourier coefficient of the heat equation in the exterior of a circle of radius $r_0$ with a zero boundary. But our study covers any real $k$ as well as complex one.

Laplace transform
$$
 \hat \omega(\tau,r )=\int_0^\infty e^{-\tau t}w(t,r) d\tau
$$
reduces (\ref{wkk:omegaeqpolar})-(\ref{wkk:omegaeqinit}) to Poison equation with parameter $\tau \in \CM$:
\begin{equation}\label{poison}
\Delta_k \hat \omega - \tau \hat \omega  = - f(r),
\end{equation}
or
\begin{equation} \label{besselrhs}
 \frac{\partial^2 \hat \omega}{\partial r^2} + \frac1r \frac{\partial \hat \omega}{\partial r} - \Big(\frac{k^2}{r^2} + \tau \Big) \hat \omega = -f(r).
\end{equation}
 
Homogeneous Bessel equation (\ref{besselrhs}) with zero right-hand side has two basis solutions -  modified Bessel functions of the first and second kind $I_k(\sqrt{\tau}r)$, $K_k(\sqrt{\tau}r)$.

Using
\begin{equation*}\label{wronsk}
I_k(r) \frac{\partial K_k(r)}{\partial r}-
K_k(r) \frac{\partial I_k(r)}{\partial r}=-r^{-1}.
\end{equation*}
we will have the particular solution of non-homogeneous Bessel equation (\ref{besselrhs}):
$$K_k(\sqrt{\tau}r)\int_{r_0}^r f(s) I_k(\sqrt{\tau}s) s ds \nonumber \\ + 
I_k(\sqrt{\tau}r)\int_r^{\infty} f(s) K_k(\sqrt{\tau}s) s ds. $$

Finally, excluding exponentially growing function $I_k(\sqrt{\tau}r)$ from (\ref{besselrhs}) using boundary condition (\ref{wkk:omegaeqdirichlet}) we have the formula: 
\begin{align} \label{directsol}
 \hat \omega(\tau,r) = - \frac{K_k(\sqrt{\tau}r)I_k(\sqrt{\tau}r_0)}{K_k(\sqrt{\tau}r_0)} \int_{r_0}^{\infty} f(s) K_k(\sqrt{\tau}s) s ds \\ + K_k(\sqrt{\tau}r)\int_{r_0}^r f(s) I_k(\sqrt{\tau}s) s ds \nonumber  + 
I_k(\sqrt{\tau}r)\int_r^{\infty} f(s) K_k(\sqrt{\tau}s) s ds.
\end{align}

In order to obtain the solution $w(t,r)$ we apply inverse Laplace transform
$$
w(t,r) = \frac1{2\pi i} \int_{\Gamma_\eta}e^{\tau t}\hat \omega(\tau,r) d\tau,
$$
where $\Gamma_\eta = \{ \tau \in \CM, \operatorname{Re}\tau = \eta \}$, $\eta$ - fixed arbitrary positive real number.

We change contour of integration from $\Gamma_\eta$ to Hankel-type contour $\Gamma_{-\eta, \varepsilon}^-$ with some $\varepsilon>0$, where
\begin{eqnarray}
\Gamma_{-\eta, \varepsilon}^- = ( -\eta -i\infty, -\eta -i\varepsilon] \cup
[-\eta - i\varepsilon, - i\varepsilon]     
\cup [- i\varepsilon, i\varepsilon] \nonumber \\
\cup [i\varepsilon, -\eta + i\varepsilon] 
\cup [-\eta + i\varepsilon, -\eta + i\infty). \nonumber
\end{eqnarray}

Define oriented contour $\gamma_{\pm,\eta} = [-\eta,0] \cup [0,-\eta]$. Then, going $\varepsilon \to 0$, $\eta \to \infty$ we will have:
\begin{align} \label{wkk:wksol}
&w(t,r) =  
\frac1{2\pi i} \left ( \int_{-\infty}^0 e^{\tau t}\hat \omega(\tau-i0,r)d\tau 
+ \int_0^{-\infty} e^{\tau t}\hat \omega(\tau+i0,r)d\tau \right )\\
&=
\frac1{\pi i}  \int_0^\infty e^{-\lambda^2 t} \left ( w(-\lambda^2-i0,r)  - w(-\lambda^2+i0,r) \right )
 \lambda \mathrm{d}\lambda  \nonumber.
\end{align} 

$I_k,K_k$ on imaginary line is given by formulas (see \cite{W}): 
\begin{align} \label{wkk:bescont1}
&I_k(-i\lambda r)=e^{-\frac {\pi i k}2} J_k(\lambda r),  \\  
&I_k(i\lambda r)=e^{-\frac {\pi i k}2}  J_k(-\lambda r),  \\ 
&K_k(-i\lambda r) = \frac{\pi i}2 e^{\frac {\pi i k}2} H_k^{(1)}(\lambda r),  \\ 
&K_k(i\lambda r) =\frac{\pi i}2 e^{\frac {\pi i k}2} H_{k}^{(1)}(-\lambda r),\\
&H_{k}^{(1)}(-\lambda r)=-e^{-\pi i k}H_{k}^{(2)}(\lambda r),\\
&J_{k}(-\lambda r)=e^{\pi i k} J_{k}(\lambda r).\label{wkk:bescont2}
\end{align}
Here $H_k^{(1)}$, $H_{k}^{(2)}$  - Hankel functions (see \cite{BE}, \cite{W}):
\begin{align*} 
&H_k^{(1)}(\lambda r) = J_k(\lambda r) + iY_k(\lambda r) \\
&H_{k}^{(2)}(\lambda r) =
J_{k}(\lambda r) - iY_{k}(\lambda r).
\end{align*}

We decompose function $\hat \omega(\tau,r)$ in (\ref{wkk:wksol}):
$$
 \hat \omega(\tau,r) = -G_{k,1}(\tau,r)+ G_{k,2}(\tau,r), 
$$ 
where $G_{k,1}(\tau,r)$, $G_{k,2}(\tau,r)$ are defined as 
\begin{align*}
G_{k,1}(\tau,r) &=  \frac{K_k(\sqrt{\tau}r)I_k(\sqrt{\tau}r_0)}{K_k(\sqrt{\tau}r_0)} \int_{r_0}^{\infty} f(s) K_k(\sqrt{\tau}s) s ds, \nonumber \\
G_{k,2}(\tau,r) &= K_k(\sqrt{\tau}r)\int_{r_0}^r f(s) I_k(\sqrt{\tau}s) s ds \nonumber \\ 
&+ I_k(\sqrt{\tau}r)\int_r^{\infty} f(s) K_k(\sqrt{\tau}s) sds. \nonumber
\end{align*}

Then with help of (\ref{wkk:bescont1})-(\ref{wkk:bescont2}) we will have
\begin{align*}
&G_{k,1}(-\lambda^2-i0,r)  - G_{k,1}(-\lambda^2+i0,r) = \\
&=\int_{r_0}^\infty \Big ( \frac{K_k(-i\lambda r) I_{k}(-i\lambda r_0)K_k(-i\lambda s)}{K_{k}(-i\lambda r_0)} \\
&- \frac{K_k(i\lambda r) I_{k}(i\lambda r_0)K_k(i\lambda s)}{K_{k}(i\lambda r_0)} \Big ) f(s)  s ds \nonumber \\
&=\frac{\pi i}2 \int_{r_0}^\infty \Big ( \frac{H_k^{(1)}(\lambda r) J_{k}(\lambda r_0)H_k^{(1)}(\lambda s)}{H_{k}^{(1)}(\lambda r_0)} \\
&- \frac{H_{k}^{(1)}(-\lambda r) J_{k}(-\lambda r_0)H_{k}^{(1)}(-\lambda s)}{H_{(k)}^{(1)}(-\lambda r_0)} \Big ) f(s)  s ds \nonumber \\
&=\frac{\pi i}2 \int_{r_0}^\infty \Big ( \frac{H_k^{(1)}(\lambda r) J_{k}(\lambda r_0)H_k^{(1)}(\lambda s)}{H_{k}^{(1)}(\lambda r_0)} \\
&+ \frac{H_{k}^{(2)}(\lambda r) J_{k}(\lambda r_0)H_{k}^{(2)}(\lambda s)}{H_{(k)}^{(2)}(\lambda r_0)} \Big ) f(s)  s ds \nonumber \\
&=\frac{\pi i}2 \int_{r_0}^\infty  \Big ( \frac{(J_k(\lambda r)+iY_k(\lambda r)) J_{k}(\lambda r_0)(J_k(\lambda s)+iY_k(\lambda s))}{J_{k}(\lambda r_0)+iY_{k}(\lambda r_0)} \\ 
& + \frac{(J_k(\lambda r)-iY_k(\lambda r)) J_{k}(\lambda r_0)(J_k(\lambda s)-iY_k(\lambda s))}{J_{k}(\lambda r_0)-iY_{k}(\lambda r_0)}  \Big ) f(s)  s ds.
\nonumber
\end{align*}

Here we expand one lemma which was proved in \cite{G2} for non-negative integer $k$ on the case of arbitrary $k\in\RM$:
\begin{lem} \label{wkk:besselrelationlem}
For any $k\in \RM$,  $r,s > 0$ modified Bessel functions $I_k,K_k$ will satisfy: 
$$
I_k(-ir)K_k(-is) - I_k(i r)K_k(is) = \pi i J_k( r)J_k( s).
$$
\end{lem}
\begin{proof}
Multiplying among themselves the following relations on Bessel functions with argument $-z$ given in \cite{W}
\begin{align*}
&I_k(-z)=e^{i\pi k} I_k(z) \\
&K_k(-z)=e^{-i\pi k} K_k(z) -i\pi I_k(z) 
\end{align*}
we will have
\begin{align*}
I_k(-ir)K_k(-is) - I_k(i r)K_k(is) = i\pi e^{i\pi k}I_k(-ia)I_k(-ib).
\end{align*}
Aplying (\ref{wkk:bescont1}) we get the statement of the lemma.

\end{proof}

From this lemma follows
\begin{align*}
G_{k,2}(-\lambda^2-i0,r)  - G_{k,2}(-\lambda^2+i0,r) = \pi i \int_{r_0}^\infty J_k(\lambda r) J_k(\lambda s) f(s) s ds,
\end{align*}
and

\begin{align*}&\hat \omega(-\lambda^2-i0,r) - \hat \omega(-\lambda^2+i0,r) \nonumber \\ 
&=-\left (G_{k,1}(-\lambda^2-i0,r)  - G_{k,1}(-\lambda^2+i0,r) \right ) \nonumber \\
&+ G_{k,2}(-\lambda^2-i0,r)  - G_{k,2}(-\lambda^2+i0,r)   
=\frac{\pi i}2 \int_{r_0}^\infty \Big ( \\ &- \frac{J_{k}(\lambda r_0)(J_k(\lambda r)+iY_k(\lambda r)) (J_k(\lambda s)+iY_k(\lambda s)) (J_{k}(\lambda r_0)-iY_{k}(\lambda r_0) )}{J_{k}(\lambda r_0)^2+Y_{k}(\lambda r_0)^2}\\
&- \frac{J_{k}(\lambda r_0)(J_k(\lambda r)-iY_k(\lambda r)) (J_k(\lambda s)-iY_k(\lambda s))(J_{k}(\lambda r_0)+iY_{k}(\lambda r_0))}{J_{k}(\lambda r_0)^2+Y_{k}(\lambda r_0)^2}
\\
&+2\frac {J_k(\lambda r) J_k(\lambda s) (J_{k}(\lambda r_0)^2+Y_{k}(\lambda r_0)^2)} {J_{k}(\lambda r_0)^2+Y_{k}(\lambda r_0)^2} \Big )  f(s) s ds.\\
\end{align*}

After a series of transformations we will have
\begin{align*}&\hat \omega(-\lambda^2-i0,r) - \hat \omega(-\lambda^2+i0,r)\\
=&\pi i \int_{r_0}^\infty  \frac{ \left (J_{k}(\lambda r_0) Y_k(\lambda r) - Y_{k}(\lambda r_0) J_k(\lambda r) \right )
 } {J_{k}(\lambda r_0)^2 + Y_{k}(\lambda r_0)^2} \\
&\times \left ( J_{k}(\lambda r_0) Y_k(\lambda s) - Y_{k}(\lambda r_0) J_k(\lambda s) \right )  f(s) s ds  
\end{align*}

Then from (\ref{wkk:wksol})
\begin{align}\label{wkk:sol}
&w(t,r) = 
\int_0^\infty  \frac{  \left (J_{k}(\lambda r_0) Y_k(\lambda r) - Y_{k}(\lambda r_0) J_k(\lambda r)  \right )e^{-\lambda^2 t}} 
{\sqrt{J_{k}(\lambda r_0)^2 + Y_{k}(\lambda r_0)^2}}  \\ 
&\times \int_{r_0}^\infty
 \frac{ \left (J_{k}(\lambda r_0) Y_k(\lambda s) - Y_{k}(\lambda r_0) J_k(\lambda s) \right )}{\sqrt{J_{k}(\lambda r_0)^2 + Y_{k}(\lambda r_0)^2}} f(s) s \ds \lambda \mathrm{d}\lambda, \nonumber
\end{align}
or
\begin{equation}\label{wkk:solw}
w(t,\cdot) = W^{-1}_{k,k}\left [ e^{-\lambda^2 t} W_{k,k}[f] \right ].
\end{equation}

\begin{thm} \label{wkk:thmheateq}
Let $f(r) \sqrt r \in L_1(r_0,\infty) \cap L_2(r_0,\infty)$ for $k \in \ZM$. Then the solution $w(t,r)\in C\left(\RM_+;L_2(r_0,\infty; r)\right )$ of (\ref{wkk:omegaeqpolar}), (\ref{wkk:omegaeqdirichlet}),(\ref{wkk:omegaeqinit}) is given by (\ref{wkk:solw}).
\end{thm}

\begin{proof} 
$\Delta_k$ generates strongly continuous semi-group in $L_2(r_0,\infty; r)$ for boundary-value problem (\ref{wkk:omegaeqpolar}), (\ref{wkk:omegaeqdirichlet}),(\ref{wkk:omegaeqinit}). Indeed, Laplace operator $\Delta=\partial^2_{x_1^2}+\partial^2_{x_2^2}$  generates strongly continuous semi-group in $L_2(\{x\in \RM^2,|x|\geq r_0\})$ (see \cite{Yosida}). Then from relation
$\Delta [f(r)e^{ik\varphi}] = e^{ik\varphi} \Delta_k f(r)$ operator $\Delta_k$ does the same in $L_2(r_0,\infty; r)$ and solution of boundary-value problem exists. Last integral in (\ref{wkk:sol}) by lemma \ref{wkk:lemest} is bounded by $1/\sqrt \lambda$. Then first integral in  (\ref{wkk:sol}) involving $e^{-\lambda^2 t}$ is well defined and the theorem is proved. 
\end{proof}

\subsection{Proof of the inversion theorem} \label{wkk:proof}

We prove theorem \ref{wkk:thm} using the initial-boundary value problem (\ref{wkk:omegaeqpolar}), (\ref{wkk:omegaeqdirichlet}),(\ref{wkk:omegaeqinit}) solved in previous subsection. All we need is to prove the validity of the limit transition $t\to 0$ in (\ref{wkk:solw}).

\begin{proof}
Let $w(t,r)$ be the solution of (\ref{wkk:omegaeqpolar}), (\ref{wkk:omegaeqdirichlet}) with initial datum $f(r)$. Since $\Delta_k$ generates strongly continuous semi-group in $L_2(r_0,\infty; r)$ then
$$
\|w(t,\cdot) - f(\cdot)\|_{L_2(r_0,\infty; r)} \to 0,~t\to 0.
$$ 

Now we prove weak convergence in $L_2(r_0,\infty; r)$
$$
w(t,\cdot) \rightharpoondown W^{-1}_{k,k}\left [W_{k,k} [f] \right ]  ,~t\to 0.
$$

It is enough to prove the weak convergence for $g(r)$ from the dense set of functions from $C_0^\infty(r_0,\infty)$. From Fubini’s theorem
$$
\left (W^{-1}_{k,k}\left [ e^{-\lambda^2 t} W_{k,k}[f] \right ], g(\cdot) \right ) = \\
\left ( e^{-\lambda^2 t}W_{k,k}[f], W_{k,k}[g] \right )=
\left ( e^{-\lambda^2 t}\hat f(\lambda), \hat g (\lambda) \right ),
$$
where $\hat f(\lambda) = W_{k,k}[f](\lambda)$, 
$\hat g(\lambda) = W_{k,k}[g] (\lambda)$.

And from lemma (\ref{wkk:lemest}) follows
\begin{align} \label{wkk:scalarfg}
\left | \left (e^{-\lambda^2 t} \hat f(\lambda), \hat g (\lambda) \right ) \right | = \left | \int_0^\infty e^{-\lambda^2 t} \hat f(\lambda) \hat g (\lambda) \lambda \dlambda \right |\leq C \|f\|_{L_1}
\int_0^\infty e^{-\lambda^2 t} |\hat g (\lambda)| \sqrt \lambda \dlambda.
\end{align}


Since $g(r)$ has compact support then for large $\lambda$
$$
\hat g(\lambda)=O\left (\frac 1{\lambda^2}  \right )
$$
and integral tails in right side of (\ref{wkk:scalarfg}) of the form
$$
\int_L^\infty \frac{e^{-\lambda^2 t}}{\lambda^\frac 32}\dlambda
$$
converge uniformly to zero when $L\to \infty$  by $t$, and integrand function ${e^{-\lambda^2 t}}/{\lambda^\frac 32}$ converges uniformly on any compact subset of $\RM_+$. So we can pass to limit $t\to 0$ in (\ref{wkk:scalarfg}). Then $w(t,\cdot) \rightharpoondown W^{-1}_{k,k}\left [W_{k,k} [f] \right ]$ and combined with $w(t,\cdot) \rightarrow f(\cdot)$ formula (\ref{wkk:inversion_formula}) holds almost everywhere.

\end{proof}

\section{Weber-Orr transform $W_{k,k\pm 1}[\cdot]$ }

In this section we provide $\Delta_k$ with such Robin-type boundary condition, that the kernel of $\Delta_k$ becomes nontrivial.
The resolvent $R(\tau) = (\Delta_k-\tau I)^{-1}$ is branching function with jump discontinuity along continuous spectrum $\sigma_c(\Delta_k) = \RM_-$. Eigen value $\tau=0$ will generate term like $R_{-1}/\tau$ in expansion of $R(\tau)$, where $R_{-1}$ is the projector onto eigen subspace $\ker(\Delta_k)$. This subspace combined with functions from continuous part of the spectrum $\sigma_c(\Delta_k)$ constitute full system in $L_2(r_0,\infty; r)$.

Transform $W_{k,k\pm 1}[\cdot]$ is defined for $k\in \RM$, $r_0>0$ as
\begin{equation}\label{int:weberorr}
W_{k,k\pm 1}[f](\lambda) = \int_{r_0}^\infty \frac{J_{k}(\lambda s)Y_{k\pm 1}(\lambda r_0) - Y_{k}(\lambda s)J_{k\pm 1}(\lambda r_0)}
{\sqrt{J_{k\pm 1}^2(\lambda r_0) + Y_{k\pm 1}^2(\lambda r_0)}}   f(s) s \ds.
\end{equation}

We define inverse operator
\begin{equation}\label{int:weberorrinv}
W^{-1}_{k,k\pm 1}[\hat f](r) = \int_{0}^\infty \frac  {J_{k}(\lambda r)Y_{k\pm 1}(\lambda r_0) - Y_{k}(\lambda s)J_{k\pm 1}(\lambda r_0)}{\sqrt{J_{k\pm 1}^2(\lambda r_0) + Y_{k\pm 1}^2(\lambda r_0)}} \hat f (\lambda) \lambda \dlambda.
\end{equation}


As we will see $\ker(W^{-1}_{k,k\pm 1})\neq \{0\}$ and Plancherel identity will be supplemented by additional term  as in (\ref{intro:PP}) and inversion formula like (\ref{wkk:inversion_formula}) for $W_{k,k\pm 1}[\cdot]$ lose its validity.

When transform $W_{k,k}[\cdot]$ helps to find the solution of heat equation with zero Dirichlet condition by formula (\ref{wkk:solw}), $W_{k,k \pm 1}[\cdot]$ solves heat equation 
\begin{equation} \label{wkk1:omegaeqpolar}
\frac{\partial w(t,r)}{\partial t} - \Delta_k w(t,r) = 0,
\end{equation}
with Robin-type boundary
\begin{equation}\label{wkk1:robin_bound}
r_0\frac{\partial w(t,r)}{\partial r}\Big|_{r=r_0} \mp k w(t,r_0) = 0,~k \in \RM.
\end{equation}

Supply this boundary problem with initial datum \begin{equation} \label{wkk1:omegaeqinit}
w(0,r) = f(r).
\end{equation}

If we had inversion formula like (\ref{wkk:inversion_formula}) then the solution of (\ref{wkk1:omegaeqpolar}), (\ref{wkk1:robin_bound}), (\ref{wkk1:omegaeqinit}) would be given by the following equality
$$
w(t,\cdot) = W^{-1}_{k,k\pm 1}\left [ e^{-\lambda^2 t} W_{k,k\pm 1}[f] \right ].
$$

But since
$$
\Delta_k[1/r^k]=0,~\Delta_k[r^k]=0
$$
then for $k>1$ the function $1/r^k$ $\in$ $\ker(\Delta_k)\subset L_2(r_0,\infty; r)$, and for $k<-1$ the function $r^k$ $\in$ $\ker(\Delta_k) \subset L_2(r_0,\infty; r)$. Both $1/r^k$ and $r^k$
 are the stationary solutions of (\ref{wkk1:omegaeqpolar}). $1/r^k$ satisfies boundary condition with sign "+" in (\ref{wkk1:robin_bound}) when $r^k$ does the same with "-". And so, these functions must be in kernel of $W_{k,k\pm 1}$ which is proved in the following lemma. 

\begin{lem} \label{wkk1:lem1}For $k>1$ function $1/r^k \in \ker(W_{k,k-1}) \subset L_2(r_0,\infty; r)$, and
for $k<-1$ $r^k \in \ker(W_{k,k-1}) \subset L_2(r_0,\infty; r)$. For $k<-1$ function $r^k \in \ker(W_{k,k+1}) \subset L_2(r_0,\infty; r)$.
\end{lem}

\begin{proof}For $k>1$ function $1/r^k \in L_2(r_0,\infty; r)$,
and for $k<-1$ function $r^k \in L_2(r_0,\infty; r)$. From relations (see \cite{BE})
\begin{align*}
&\int s^{-k+1}J_k(s)ds=s^{-k+1}J_{k-1}(s),\\
&\int s^{-k+1}Y_k(s)ds=s^{-k+1}Y_{k-1}(s),\\
&\int s^{k+1}J_k(s)ds=s^{k+1}J_{k+1}(s),\\
&\int s^{k+1}Y_k(s)ds=s^{k+1}Y_{k+1}(s)
\end{align*}
follows
\begin{align*}
&W_{k,k-1}[1/r^k](\lambda) = \\ & \int_{r_0}^\infty \left ( J_{k}(\lambda s)Y_{k-1}(\lambda r_0) - Y_{k}(\lambda s)J_{k-1}(\lambda r_0) \right ) \left ( 1/s^k\right ) s\ds = 0,~k>1, \\
&W_{k,k+1}[r^k](\lambda) = \\ & \int_{r_0}^\infty \left ( J_{k}(\lambda s)Y_{k-1}(\lambda r_0) - Y_{k}(\lambda s)J_{k-1}(\lambda r_0) \right ) \left ( s^k\right ) s\ds = 0,~k<-1.
\end{align*}
\end{proof}

The main result of this section is the following theorem. It involves inversion formulas and  Plancherel-Parseval type identity like (\ref{intro:PP}).

\begin{thm}\label{wkk1:thm}Let $f(r) \sqrt r \in L_1(r_0,\infty)\cap L_2(r_0,\infty)$, $r_0>0$, $k \in \RM$. Then the associated Weber-Orr transforms (\ref{int:weberorr}), (\ref{int:weberorrinv}) satisfy  almost everywhere 
\begin{align*}
f(r) &= W^{-1}_{k,k-1}\left [W_{k,k-1} [f] \right ](r), k\leq 1, \\
f(r) &= W^{-1}_{k,k-1}\left [W_{k,k-1} [f] \right ](r) +  \frac {2(k-1) r_0^{2(k-1)}} {r^k} \int_{r_0}^\infty s^{-k+1} f(s) \ds,k>1,\\
f(r) &= W^{-1}_{k,k+1}\left [W_{k,k+1} [f] \right ](r), k\geq -1, \\
f(r) &= W^{-1}_{k,k+1}\left [W_{k,k+1} [f] \right ](r) -  \frac {2(k+1) r^{k}} {r_0^{2(k+1)}} \int_{r_0}^\infty s^{-k+1} f(s) \ds,k<-1.
\end{align*}
and the following Plancherel-Parseval equity holds:
\begin{align*}
\|f\|_{L_2(r_0,\infty; r)}^2 &= \| W_{k,k-1}[f] \|_{L_2(0,\infty; r) }^2 , k\leq 1, \\
\|f\|_{L_2(0,\infty; r) }^2 &= \| W_{k,k-1}[f] \|_{L_2(0,\infty; r) } ^2 \\
&+ {2(k-1)} r_0^{2(k-1)} \left (f,1/r^k \right )^2_{L_2(r_0,\infty; r)}, k> 1,\\
\|f\|_{L_2(r_0,\infty; r)}^2 &= \| W_{k,k+1}[f] \|_{L_2(0,\infty; r) }^2 , k\geq -1, \\
\|f\|_{L_2(0,\infty; r) }^2 &= \| W_{k,k-1}[f] \|_{L_2(0,\infty; r) } ^2 \\
&- \frac{2(k+1)}{r_0^{2(k+1)}} \left (f,r^k \right )^2_{L_2(r_0,\infty; r)}, k< -1.
\end{align*}
\end{thm}

From this theorem  immediately follows
\begin{cor}[Bessel-type inequality]For $f(r)$, $f(r) \sqrt r \in L_1(r_0,\infty)\cap L_2(r_0,\infty)$, $k \in \RM$, $r_0>0$ Bessel-type inequality holds:
$$
\| W_{k,k\pm 1}[f] \|_{L_2(0,\infty; r) }^2 \leq \|f\|_{L_2(r_0,\infty; r)}^2. 
$$
\end{cor}

\subsection{Initial-boundary value problem}

Due to symmetry of transforms $W_{k,k-1}$, $W_{k,k+1}$ further study will be carried out only for $W_{k,k-1}$ with boundary
\begin{equation}\label{wkk1:robin_bound1}
r_0\frac{\partial w(t,r)}{\partial r}\Big|_{r=r_0} + k w(t,r_0) = 0,~k \in \RM.
\end{equation}

In this section we will solve the initial boundary value problem (\ref{wkk1:omegaeqpolar}), (\ref{wkk1:omegaeqinit}),  (\ref{wkk1:robin_bound1}) reducing it to elliptic equation with parameter $\tau \in \CM$.

Laplace transform
$$
 \hat \omega(\tau,r)=\int_0^\infty e^{-\tau t}w(t,r) d\tau
$$
reduces (\ref{wkk1:omegaeqpolar}) to 
\begin{equation}\label{wkk1:poison}
\Delta_k \hat \omega - \tau \hat \omega  = - f(r),
\end{equation}
or
\begin{equation} \label{wkk1:besselrhs}
 \frac{\partial^2 \hat \omega}{\partial r^2} + \frac1r \frac{\partial \hat \omega}{\partial r} - \Big(\frac{k^2}{r^2} + \tau \Big) \hat \omega = -f(r).
\end{equation}
 
Particular solution of this equation was established in previous section:
$$\hat w_p(\tau, r) = K_k(\sqrt{\tau}r)\int_{r_0}^r f(s) I_k(\sqrt{\tau}s) s ds \nonumber \\ + 
I_k(\sqrt{\tau}r)\int_r^{\infty} f(s) K_k(\sqrt{\tau}s) s ds. $$

Since
$$
\frac{\hat w_p}{\partial r}(\tau, r_0) = \sqrt \tau I_k'(\sqrt \tau r_0) \int_{r_0}^\infty f(s) K_k(\sqrt{\tau}s) s ds,
$$
then
\begin{align*}
r_0{\frac{\partial \hat w_p}{\partial r }} (\tau,r_0) + k \hat w_p(\tau,r_0)  =  \sqrt \tau r_0 I_k'(\sqrt \tau r_0) \int_{r_0}^\infty f(s) K_k(\sqrt{\tau}s) s ds \\ + k I_k(\sqrt{\tau}r_0)
 \nonumber  \int_{r_0}^\infty f(s) K_k(\sqrt{\tau}s) s ds.
\end{align*}

From differentiation relation on Bessel functions 
\begin{eqnarray*}
I_{k-1}(r)={\frac {k }{r}}I_{k }(r)+{\frac {dI_{k }}{dr}}(r) 
\end{eqnarray*}
we will have
$$
\sqrt \tau r_0 I_k'(\sqrt \tau r_0) + k I_k(\sqrt{\tau}r_0) = \sqrt \tau r_0 I_{k-1}(\sqrt \tau r_0).
$$
and
\begin{align*}
r_0{\frac{\partial \hat w_p}{\partial r }} (\tau,r_0) + k \hat w_p(\tau,r_0) 
 =  r_0 \sqrt \tau  I_{k-1}(\sqrt \tau r_0) \int_{r_0}^\infty f(s) K_k(\sqrt{\tau}s) s ds.
\end{align*}
From
\begin{align*}
K_{k-1}(r)=-{\frac {k }{r}}K_{k }(r) -{\frac {dK_{k }}{dr}}(r)
\end{align*}
follows
\begin{align*}
r_0{\frac{\partial K_k(\sqrt \tau r)}{\partial r }} (\sqrt \tau r_0) + k K_k(\sqrt \tau r_0) =  -r_0 \sqrt \tau K_{k-1}(\sqrt \tau r_0).
\end{align*}

And then after excluding exponentially growing function $I_k(\sqrt{\tau}r)$ the general solution of (\ref{wkk1:besselrhs}) is as follows
\begin{align} \label{wkk1:omegasol}
 \hat \omega(\tau,r) = \frac{K_{k}(\sqrt{\tau}r) I_{k-1}(\sqrt \tau r_0)}{K_{k-1}(\sqrt \tau r_0)} \int_{r_0}^\infty f(s) K_{k}(\sqrt{\tau}s) s ds\\+ K_{k}(\sqrt{\tau}r)\int_{r_0}^r f(s) I_{k}(\sqrt{\tau}s) s ds  + 
I_{k}(\sqrt{\tau}r)\int_r^{\infty} f(s) K_{k}(\sqrt{\tau}s) s ds.  \nonumber 
\end{align}

The solution $w(t,x)$ of the heat equation is expressed by inverse Laplace transform
$$
w(t,r) = \frac1{2\pi i} \int_{\Gamma_\eta}e^{\tau t}\hat \omega(\tau,r)d\tau,
$$
where $\Gamma_\eta = \{ \tau \in \CM, \operatorname{Re}\tau = \eta \}$, $\eta$ - fixed arbitrary positive real number.

By analogy with derivation of $W_{k,k}[\cdot]$ in previous section since $\sqrt \tau$ involved in (\ref{wkk1:omegasol}) is a branching function, then we change contour of integration from $\Gamma_\eta$ to $\Gamma_{-\eta, \varepsilon}^-$ with some $\varepsilon>0$, where
\begin{eqnarray}
\Gamma_{-\eta, \varepsilon}^- = ( -\eta -i\infty, -\eta -i\varepsilon] \cup
[-\eta - i\varepsilon, - i\varepsilon]     
\cup [- i\varepsilon, i\varepsilon] \nonumber \\
\cup [i\varepsilon, -\eta + i\varepsilon] 
\cup [-\eta + i\varepsilon, -\eta + i\infty). \nonumber
\end{eqnarray}

From (\ref{wkk1:poison}) function $\hat \omega(\tau,r)$ can be expressed via resolvent $R(\tau) = (\Delta_k-\tau I)^{-1}$:
$$
\hat \omega  = - R(\tau) f(r).
$$

At this point there will be a significant difference from the case of derivation of $W_{k,k}[\cdot]$. From Lemma \ref{wkk1:lem1} since $1/r^k \in \ker(W_{k,k-1})$ then  $R(\tau)$ has nonzero residue at $\tau=0$ for $k>1$ (see \cite{Kato}). If we define oriented contour $\gamma_{\pm,\eta} = [-\eta,0] \cup [0,-\eta]$, then when going $\varepsilon \to 0$, $\eta \to \infty$ we will have the equality:
\begin{align} 
&w(t,r) = 
\frac1{2\pi i} \int_{-\infty}^0 e^{\tau t}\hat \omega(\tau-i0,r)d\tau \nonumber \\  
&+ \frac1{2\pi i}   \int_0^{-\infty} e^{\tau t}\hat \omega(\tau+i0,r)d\tau 
+  \operatorname{res}\limits_{\tau=0}[\hat \omega(\tau,r)] ,  \nonumber
\end{align}
where 
$$
{\mathrm  {res}}_{\tau=0}\,[\hat \omega(\tau,r)]=
\begin{cases} 
 0, k\leq 1, \\
 \lim\limits_{\rho \to 0} {1 \over {2\pi i}}\int \limits _{{|\tau|=\rho }}\!\hat \omega(\tau,r)\,d\tau, k>1.
 \end{cases}
$$ 
 
Set $\tau = -\lambda^2$. Then, we will have  
\begin{eqnarray} \label{wkk1:wsolwithresidue}
w(t,r) = \nonumber 
\frac1{\pi i} \int_0^\infty e^{-\lambda^2 t} \left ( w(-\lambda^2-i0,r)  - w_k(-\lambda^2+i0,r) \right )
 \lambda \mathrm{d}\lambda  \nonumber \\ +  \operatorname{res}\limits_{\tau=0}[\hat \omega(\tau,r)].
\end{eqnarray}

We decompose function $\hat \omega(\tau,r)$ in (\ref{wkk1:omegasol}) as
$$
 \hat \omega(\tau,r) = G_{k,1}(\tau,r)+ G_{k,2}(\tau,r), 
$$ 
where $G_{k,1}(\tau,r)$, $G_{k,2}(\tau,r)$ are defined as 
\begin{align*}
G_{k,1}(\tau,r) &= \frac{K_k(\sqrt{\tau}r) I_{k-1}(\sqrt \tau r_0)}{K_{k-1}(\sqrt \tau r_0)} \int_{r_0}^\infty f(s) K_k(\sqrt{\tau}s) s ds, \nonumber \\
G_{k,2}(\tau,r) &= K_k(\sqrt{\tau}r)\int_{r_0}^r f(s) I_k(\sqrt{\tau}s) s ds \nonumber \\ 
&+ I_k(\sqrt{\tau}r)\int_r^{\infty} f(s) K_k(\sqrt{\tau}s) sds. \nonumber
\end{align*}

With help of (\ref{wkk:bescont1})-(\ref{wkk:bescont2}) we will have
\begin{align*}
&G_{k,1}(-\lambda^2-i0,r)  - G_{k,1}(-\lambda^2+i0,r) = \\
&=\int_{r_0}^\infty \Big ( \frac{K_k(-i\lambda r) I_{k-1}(-i\lambda r_0)K_k(-i\lambda s)}{K_{k-1}(-i\lambda r_0)} \\
&- \frac{K_k(i\lambda r) I_{k-1}(i\lambda r_0)K_k(i\lambda s)}{K_{k-1}(i\lambda r_0)} \Big ) f(s)  s ds \nonumber \\
&=-\frac{\pi i}2 \int_{r_0}^\infty \Big ( \frac{H_k^{(1)}(\lambda r) J_{k-1}(\lambda r_0)H_k^{(1)}(\lambda s)}{H_{k-1}^{(1)}(\lambda r_0)} \\
&-  \frac{H_k^{(1)}(-\lambda r) J_{k-1}(-\lambda r_0)H_k^{(1)}(-\lambda s)}{H_{k-1}^{(1)}(-\lambda r_0)}
\Big ) f(s)  s ds \nonumber \\
&=-\frac{\pi i}2 \int_{r_0}^\infty \Big ( \frac{H_k^{(1)}(\lambda r) J_{k-1}(\lambda r_0)H_k^{(1)}(\lambda s)}{H_{k-1}^{(1)}(\lambda r_0)} \\
&+  \frac{H_k^{(2)}(\lambda r) J_{k-1}(\lambda r_0)H_k^{(2)}(\lambda s)}{H_{k-1}^{(2)}(\lambda r_0)}
\Big ) f(s)  s ds \nonumber \\
&=-\frac{\pi i}2 \int_{r_0}^\infty  \Big ( \frac{(J_k(\lambda r)+iY_k(\lambda r)) J_{k-1}(\lambda r_0)(J_k(\lambda s)+iY_k(\lambda s))}{J_{k-1}(\lambda r_0)+iY_{k-1}(\lambda r_0)} \\ 
& + \frac{(J_k(\lambda r)-iY_k(\lambda r)) J_{k-1}(\lambda r_0)(J_k(\lambda s)-iY_k(\lambda s))}{J_{k-1}(\lambda r_0)-iY_{k-1}(\lambda r_0)}  \Big ) f(s)  s ds.
\nonumber
\end{align*}

From lemma \ref{wkk:besselrelationlem} follows 
\begin{align*}
G_{k,2}(-\lambda^2-i0,r)  - G_{k,2}(-\lambda^2+i0,r) = \pi i \int_{r_0}^\infty J_k(\lambda r) J_k(\lambda s) f(s) s ds,
\end{align*}
and
\begin{align*}&\hat \omega(-\lambda^2-i0,r) - \hat \omega(-\lambda^2+i0,r) \nonumber \\ 
&=G_{k,1}(-\lambda^2-i0,r)  - G_{k,1}(-\lambda^2+i0,r) \nonumber \\
&+ G_{k,2}(-\lambda^2-i0,r)  - G_{k,2}(-\lambda^2+i0,r)   
=-\frac{\pi i}2 \int_{r_0}^\infty \Big ( \\ & \frac{J_{k-1}(\lambda r_0)(J_k(\lambda r)+iY_k(\lambda r)) (J_k(\lambda s)+iY_k(\lambda s)) (J_{k-1}(\lambda r_0)-iY_{k-1}(\lambda r_0) )}{J_{k-1}(\lambda r_0)^2+Y_{k-1}(\lambda r_0)^2}\\
&+ \frac{J_{k-1}(\lambda r_0)(J_k(\lambda r)-iY_k(\lambda r)) (J_k(\lambda s)-iY_k(\lambda s))(J_{k-1}(\lambda r_0)+iY_{k-1}(\lambda r_0))}{J_{k-1}(\lambda r_0)^2+Y_{k-1}(\lambda r_0)^2}
\\
&-2\frac {J_k(\lambda r) J_k(\lambda s) (J_{k-1}(\lambda r_0)^2+Y_{k-1}(\lambda r_0)^2)} {J_{k-1}(\lambda r_0)^2+Y_{k-1}(\lambda r_0)^2} \Big )  f(s) s ds.\\
\end{align*}

After opening brackets in previous formula we will have
\begin{align*}&\hat \omega(-\lambda^2-i0,r) - \hat \omega(-\lambda^2+i0,r)\\
=&\pi i \int_{r_0}^\infty  \frac{ \left (J_{k-1}(\lambda r_0) Y_k(\lambda r) - Y_{k-1}(\lambda r_0) J_k(\lambda r) \right )
 } {J_{k-1}(\lambda r_0)^2 + Y_{k-1}(\lambda r_0)^2} \\
&\times \left ( J_{k-1}(\lambda r_0) Y_k(\lambda s) - Y_{k-1}(\lambda r_0) J_k(\lambda s) \right )  f(s) s ds 
\end{align*}

Then we substitute the obtained jump discontinuity $\hat \omega(-\lambda^2-i0,r) - \hat \omega(-\lambda^2+i0,r)$ into (\ref{wkk1:wsolwithresidue}):
\begin{align*}
&w(t,r) = \nonumber 
 \int_0^\infty  \frac{ \left (J_{k-1}(\lambda r_0) Y_k(\lambda r) - Y_{k-1}(\lambda r_0) J_k(\lambda r) \right ) } 
{\sqrt{J_{k-1}(\lambda r_0)^2 + Y_{k-1}(\lambda r_0)^2}} \\ 
& \times \left ( \int_{r_0}^\infty \frac{ \left (J_{k-1}(\lambda r_0) Y_k(\lambda s) - Y_{k-1}(\lambda r_0) J_k(\lambda s) \right ) } 
{\sqrt{J_{k-1}(\lambda r_0)^2 + Y_{k-1}(\lambda r_0)^2}}  f(s) s ds \right )  e^{-\lambda^2 t} \lambda \mathrm{d}\lambda 
\\
&+ \operatorname{res}\limits_{\tau=0}[\hat \omega(\tau,r)]. 
\end{align*}

Now we will find the residues of $\hat \omega(\tau,r)$ at $\tau=0$. 

\begin{lem} The residues of $\hat \omega(\tau,r)$ defined by (\ref{wkk1:omegasol}) for $k>1$ at $\tau=0$  are given by
$$\operatorname{res}\limits_{\tau=0}[\hat \omega(\tau,r)] = \frac {2(k-1) r_0^{2k-2}} {r^{k}} \int_{r_0}^\infty s^{-k+1} f(s) \ds,k>1.
$$ 
\end{lem}

\begin{proof}We will use the asymptotic form for small arguments $0 < |z| \leq \sqrt{k+1}$ (see \cite{BE}):
\begin{align*}
 I_{k }(z) & \sim {\frac {1}{\Gamma (k +1)}}\left({\frac {z}{2}}\right)^{k },\\K_{k }(z)&\sim {\frac {\Gamma (k )}{2}}\left({\dfrac {2}{z}}\right)^{k }.
\end{align*}

For $k>1$ 
$$
G_{k,1}(\tau,r) \sim  \frac {2(k-1) r_0^{2k-2}} {\tau r^k} \int_{r_0}^\infty s^{-k+1} f(s) \ds
$$
and
$$
\operatorname{res}\limits_{\tau=0}[G_{k,1}(\tau,r)]  = \frac {2(k-1) r_0^{2k-2}} {r^k} \int_{r_0}^\infty s^{-k+1} f(s) \ds.
$$

Since
\begin{align*}
K_{k }(\sqrt \tau r)I_{k }(\sqrt \tau s)&\sim \frac s{2kr}
\end{align*}
the residues of $G_{k,2}(\tau,r)$ are equal to zero and 
$$\operatorname{res}\limits_{\tau=0}[\hat \omega(\tau,r)] = \operatorname{res}\limits_{\tau=0}[\hat G_{k,1}(\tau,r)].$$ 
\end{proof}

Finally, the solution of the heat equation (\ref{wkk1:omegaeqpolar}),  supplied with Robin condition (\ref{wkk1:robin_bound1}) and initial datum $f(r)$ is represented as
\begin{align}
{\rm for~} k \leq 1:\nonumber  \\
&w(t,r) = \label{wkk1:heateqdirectsol1}
\int_0^\infty  \frac{  R_{k,k-1}(\lambda, r)} 
{J_{k-1}(\lambda r_0)^2 + Y_{k-1}(\lambda r_0)^2} \\ 
&\times \left ( \int_{r_0}^\infty R_{k,k-1}(\lambda, s)  f(s) s \mathrm{d} s \right )    e^{-\lambda^2 t} \lambda \mathrm{d}\lambda \nonumber,\\
{\rm for~} k > 1: \nonumber \\
&w(t,r) = \label{wkk1:heateqdirectsol2}
\int_0^\infty  \frac{  R_{k,k-1}(\lambda, r)} 
{J_{k-1}(\lambda r_0)^2 + Y_{k-1}(\lambda r_0)^2} \\ 
&\times \left ( \int_{r_0}^\infty R_{k,k-1}(\lambda, s)  f(s) s \mathrm{d} s \right )    e^{-\lambda^2 t} \lambda \mathrm{d}\lambda \nonumber \\
&+ \frac {2(k-1) r_0^{2k-2}} {r^{k}} \int_{r_0}^\infty s^{-k+1} f(s) \ds \nonumber
\end{align}

These relations lead to the following  
\begin{thm} \label{wkk1:thmheateq}
Let $f(r) \sqrt r \in L_1(r_0,\infty)$, $r_0>0$, $k \in \ZM$. Then the solution $w(t,r)\in C\left (\RM_+;L_2(r_0,\infty; r) \right )$ of (\ref{wkk1:omegaeqpolar}), (\ref{wkk1:omegaeqinit}),(\ref{wkk1:robin_bound1}) is given by 
\begin{align}
{\rm for~} k \leq 1:\nonumber  \\
&w(t,r) = W^{-1}_{k,k-1}\left [ e^{-\lambda^2 t} W_{k,k-1}[f] \right ], \label{wkk1:heatsol1}\\
{\rm for~} k > 1: \nonumber \\
&w(t,r) = W^{-1}_{k,k-1}\left [ e^{-\lambda^2 t} W_{k,k-1}[f] \right ]\label{wkk1:heatsol2}\\
&+ \frac {2(k-1) r_0^{2k-2}} {r^{k}} \left (\frac 1 r^{k}, f \right )_{L_2(r_0,\infty; r)} \nonumber
\end{align}
\end{thm}

The proof of this theorem is similar to ones given in theorem \ref{wkk:thmheateq}.

\subsection{Proof of the inversion theorem }
\begin{proof}

For $k\leq 1$ the proof of the theorem \ref{wkk1:thm} repeats word by word the proof of theorem \ref{wkk:thm} in section \ref{wkk:proof}. We must prove the validity of the limit transition $t\to 0$ in (\ref{wkk1:heatsol1}). Let $w(t,r)$ be solution of (\ref{wkk1:omegaeqpolar}), (\ref{wkk1:robin_bound1}) with initial datum $f(r)$. $\Delta_k$ generates strongly continuous semi-group in $L_2(r_0,\infty; r)$ and 
$$
\|w(t,\cdot) - f(\cdot)\| \to 0,~t\to 0.
$$ 

Then as in section \ref{wkk:proof} one could get
$$
w(t,\cdot) \rightharpoondown W^{-1}_{k,k-1}\left [W_{k,k-1} [f] \right ],~t\to 0
$$
and so almost everywhere holds
$$f(\cdot)=W^{-1}_{k,k-1}\left [W_{k,k-1} [f] \right ]\|.$$

For $k>1$ we have addition stationary term in (\ref{wkk1:heatsol2}) and also can pass to limit $t\to 0$ obtaining the required inversion formula.
\end{proof}


\end{document}